\documentclass{amsart}
\usepackage{mathrsfs, graphicx, multirow, amssymb}

\newcounter{counter}
\newtheorem{Lemma}[counter]{Lemma}
\newtheorem{Thm}[counter]{Theorem}

\newtheorem{Prop}[counter]{Proposition}
\newtheorem{Conj}{Conjecture}
\newtheorem*{Ex}{Example}
\theoremstyle{definition}
\newtheorem{Def}{Definition}
\theoremstyle{remark}

\title[Graded polynomial identities for $UT_3^{(-)}$]{Graded polynomial identities for the Lie algebra of upper triangular matrices of order 3}
\author{Felipe Y. Yasumura}
\address{Department of Mathematics, Instituto de Matem\'atica e Estat\'istica, Universidade de S\~ao Paulo, SP, Brazil}
\email{fyyasumura@ime.usp.br}


\begin{document}
\maketitle
\begin{abstract}
We compute the graded polynomial identities and its graded codimension sequence for the elementary gradings of the Lie algebra of upper triangular matrices of order 3.
\end{abstract}

\section{Introduction}
Nowadays very many algebraists are interested in algebras with polynomial identities (also called PI algebras). Two main directions of investigation arise in the PI theory. One of them consists in the quantitative analysis, and the other seeks for the intrinsic properties of such algebras. For the first direction one computes the polynomial identities and related numerical invariants of a given algebra. Such computations have been fundamental for guiding the insight and development of the theory of algebras with polynomial identities. In this way the results obtained in the first direction have been decisive in the development of the whole theory of PI algebras.

In this paper we are specifically interested in the graded polynomial identities of upper triangular matrices, viewed as a Lie algebra. The polynomial identities of the triangular matrices, as an associative or a Lie algebras, are well known and easily computed (see, for instance, \cite{GZ}). However, the problem of describing the polynomial identities for the same algebra becomes way harder if one wishes to include some extra structure. For example, polynomial identities of the triangular matrices with involution are known only for matrices of size 2 and 3 \cite{VKS}. As another extra structure, the polynomial identities with derivations and some of their numerical invariants were computed for upper triangular matrices of size 2 \cite{GR2018}. Hence this algebra is challenging whenever one considers its polynomial identities with extra structure, in contrast to the fact that the ordinary case is relatively easy.

We investigate the polynomial identities in the graded sense. As an associative algebra, all gradings and graded polynomial identities were understood, see \cite{VinKoVa2004,DimasEv, VaZa2007}. The characterization from \cite{VinKoVa2004} gave us the insight to deal with the Lie case. However, we were able to compute the graded polynomial identities for triangular matrices of size 3 only. The graded polynomial identities of $\mathrm{UT}_n$ as a commutative algebra over a field of characteristic 2 and their Specht property were obtained in \cite{Manu}.

It is worth noting that the Jordan case is even harder. Its polynomial identities, even the ordinary ones, are known only for triangular matrices of size $2$ \cite{KMa,DimasSa}. Moreover the Specht property for the corresponding ideals of graded identities is obtained for the same algebra of size 2, for all gradings \cite{CMaS}. Thus for the Jordan case, which is otherwise natural, we have much less information about its polynomial identities.

In this paper, we compute the graded polynomial identities for gradings on the algebra of upper triangular matrices of order $3$, as a Lie algebra. Studying this case, we raise several conjectures and counterexamples for addressing the general case of matrices of an arbitrary size. We were also able to compute the graded codimension sequence for each elementary grading, and we deduce some interesting relations among some of these gradings.

\section{Preliminaries}
All algebras and vector spaces in this paper will be over a field $\mathbb{F}$ of characteristic zero. We denote by $G$ an abelian group with multiplicative notation and neutral element $1$. We use additive notations for the cyclic groups $\mathbb{Z}$ and $\mathbb{Z}_m$.

A $G$-grading on $UT_n$ is called \textit{elementary} if all matrix units $e_{ij}$ are homogeneous in the grading. Every $G$-grading on $UT_n$ is, up to a graded isomorphism, elementary \cite{VaZa2007}. Also the classification of elementary gradings on $UT_n$ was given in \cite{VinKoVa2004} where the authors computed the corresponding graded identities as well.

We introduce some notation. The following facts can be found in \cite{VinKoVa2004}. We recall that for any elementary grading on $UT_n$ one  necessarily has $\deg e_{ii}=1$ for every $i$. Let $\eta=(g_1,g_2,\ldots,g_{n-1})\in G^{n-1}$ be any sequence. Then we obtain an elementary $G$-grading on $UT_n$ if we put $\deg e_{i,i+1}=g_i$, for each $i=1$, 2, \dots, $n-1$. Every elementary grading can be uniquely defined in this way. We denote such an elementary grading by $(UT_n,\eta)$. In \cite[Theorem 2.3]{VinKoVa2004}, the authors proved that $(UT_n,\eta)\simeq(UT_n,\eta')$ if and only if $\eta=\eta'$.

\begin{Def}[Definition 2.1 of \cite{VinKoVa2004}]
	Fix an elementary $G$-grading $\epsilon$ on $UT_n$. A sequence $\mu=(g_1,\ldots,g_m)\in G^m$ is called $\epsilon$-good if we can find strictly upper triangular matrix units $r_1$, \dots, $r_m$ such that $\deg r_i=g_i$ for every $i=1$, 2, \dots, $m$, and $r_1r_2\cdots r_m\ne0$. Otherwise we say that $\mu$ is an $\epsilon$-bad sequence.
\end{Def}
Given a sequence $\mu=(g_1,\ldots,g_m)\in G^m$, define the polynomial $f_\mu=z_1z_2\cdots z_m$ where for each $i=1$, 2, \dots, $m$,
\[
	z_i=\left\{\begin{array}{l}x_i^{(g_i)},\text{ if $g_i\ne1$},\\{}[x_{2i-1}^{(1)},x_{2i}^{(1)}],\text{ if $g_i=1$}.\end{array}\right.
\]
The relation between sequences in $G^m$ and $G$-graded identities for $UT_n$ is the following. A sequence $\mu$ is $\epsilon$-bad if and only if $f_\mu$ is a $G$-graded identity for $(UT_n,\epsilon)$ \cite[Proposition 2.2]{VinKoVa2004}.

The graded identities for $UT_n$ as an associative algebra are known.
\begin{Thm}[Theorem 2.8 of \cite{VinKoVa2004}]
	The $G$-graded identities for an elementary $G$-grading $\epsilon$ on $UT_n$ follow from all $f_\mu$ where $\mu$ runs over the $\epsilon$-bad sequences of length at most $n$.
\end{Thm}

Now using the Lie bracket $[a,b]=ab-ba$, we turn the upper triangular matrices into a Lie algebra, denoted by $UT_n^{(-)}$. We assume the bracket left normed, that is $[a_1,a_2,\ldots,a_m]=[[a_1,\ldots,a_{m-1}],a_m]$, for every $m\ge3$. We investigate the graded identities of the Lie algebra $UT_3^{(-)}$.

The classification of isomorphism classes of elementary gradings on $UT_n^{(-)}$ was obtained in \cite{pky} (see \cite{ky} as well). An elementary grading on $UT_n^{(-)}$, as in the associative case, is uniquely defined by a sequence $\eta\in G^{n-1}$. In \cite{pky} the authors proved that $(UT_n^{(-)},\eta_1)\simeq(UT_n^{(-)},\eta_2)$ if and only if $\eta_1=\eta_2$ or $\eta_1=\text{rev}\,\eta_2$. Here one denotes $\text{rev}\,(g_1,\ldots,g_m)=(g_m,\ldots,g_1)$  the reverse sequence of $(g_1,\ldots, g_m)$.

We slightly modify the definition of $\epsilon$-bad and $\epsilon$-good sequences. To this end, let \textit{tree} be sequences of elements of the group, but with parentheses. For a given elementary grading $\eta$, we define $\eta$-good tree in the same way as for sequences, but using commutators respecting the parentheses of the tree. Analogously, we define $\eta$-bad tree and length of a tree. Given a tree $\mu$, we denote $f_\mu$ the corresponding multilinear Lie monomial (in analogy with the associative case).

Here is the formal definition.
\begin{Def}
We define a \emph{tree} $\mu$ of $G$ of length $\ell(\mu)=n$ inductively as follows:
\begin{enumerate}
\item a tree of length 1 is an element of $G$,
\item a tree of length $\ell(\mu)>1$ is a pair $(\mu_1,\mu_2)$ where $\ell(\mu_1)+\ell(\mu_2)=\ell(\mu)$.
\end{enumerate}
\end{Def}

\begin{Def}
Let $\mu$ be a tree of $G$.
\begin{enumerate}
\renewcommand{\labelenumi}{(\alph{enumi})}
\item 
Let $(UT_n^{(-)},\eta)$ be an elementary $G$-grading. We define an $\eta$-good tree $\mu$ inductively as follows: if $\ell(\mu)=1$, then there exists a strict upper triangular matrix $r$ such that $\deg r=\mu$. In this case we say that $r$ is \emph{related} to $\mu$. If $\mu=(\mu_1,\mu_2)$ then we can find $r_1$ and $r_2$ related to $\mu_1$ and $\mu_2$, respectively, such that $r'=[r_1,r_2]\ne0$; in this case, we say that $r'$ is related to $\mu$. If $\mu$ is not $\eta$-good, then we call $\mu$ an $\eta$-bad tree.
\item 
We define $f_\mu$ inductively as follows: if $\ell(\mu)=1$ then $f_\mu=x_1^{(\mu)}$ whenever $\mu\ne1$, and $f_\mu=[x_1^{(1)},x_2^{(1)}]$  otherwise. If $\mu=(\mu_1,\mu_2)$ then $f_\mu=[f_{\mu_1},f_{\mu_2}]$.
\end{enumerate}
\end{Def}

\begin{Conj}\label{firstconj}
	Consider the elementary $G$-grading $(UT_n^{(-)},\eta)$. Then its graded polynomial identities follow from the $f_\mu$ where $\mu$ runs over all $\eta$-bad trees of length at most $n$.
\end{Conj}
We remark that if we consider only sequences of elements rather than trees, then the previous conjecture is not true.
\begin{Ex}
	Consider the $G$-grading on $UT_5^{(-)}$ given by $\eta=(g,g,h,h)$ where $g$, $h\in G\setminus\{1\}$ and $g\ne h$. Then $f=[[x_1^{(g)},x_2^{(h)}],[x_3^{(g)},x_4^{(h)}]]=0$ is a graded identity that does not follow from graded identities of type $[z_1,\ldots,z_m]=0$.

	\noindent\textbf{Consequence.} In general
\[
			\left\langle[x_1^{(g_1)},\ldots,x_m^{(g_m)}]\mid\text{$(g_1,\ldots,g_m)$ is a Lie bad sequence}\right\rangle^{T_G}\subsetneq T_G\left(UT_n^{(-)}\right).
\]
\end{Ex}

There is another way to equivalently state the previous conjecture. We call $m$ a \textit{special-monomial identity} if $m$ is a graded polynomial identity and $m$ is a monomial in the variables $z_1$, $z_2$, \dots, $z_t$  where each $z_i=x_i^{(g_i)}$, for some $g_i\ne1$, or $z_i=[x_{2i-1}^{(1)},x_{2i}^{(1)}]$. We say that the length of $m$ is $t$, the number of variables $z_i$. The conjecture is then the following:
\begin{Conj}\label{secondConj}
The graded polynomial identities of an elementary $G$-grading $(UT_n^{(-)},\eta)$ follow from its special-monomial identities of length at most $n$.
\end{Conj}

Finally we record several well known facts concerning the identities of $UT_n$ and the corresponding numerical invariants. The ordinary polynomial identities of $UT_n^{(-)}$ follow from the polynomial $[[x_1,x_2],[x_3,x_4],\ldots,[x_{2n-1},x_{2n}]]$. Also, its ordinary codimension sequence was computed by Petrogradskii, see \cite{petro1997}. For the particular case $n=3$, one has
\[
c_m(UT_3^{(-)})=(m^2-5m+4)\cdot2^{m-3}+2m-2,\text{ for $m\ge2$},
\]
and $c_1(UT_3^{(-)})=1$.

In what follows we list all gradings on $UT_3^{(-)}$ and compute a basis for its graded polynomial identities and the corresponding graded codimension sequences.



\section{Elementary gradings of $UT_3^{(-)}$}
\subsection{Notation and preliminaries}
We denote the variables of trivial degree by $y_i=x_i^{(1)}$, for every $i\in\mathbb{N}$. We call \textit{1-convenient monomial} a monomial of the form $[y_{i_1},y_{i_2},\ldots,y_{i_m}]$ where $i_1>i_2$ and $i_2<i_3<\cdots<i_m$.

The elementary gradings on $UT_n^{(-)}$, up to a graded isomorphism, were described in \cite{pky}. Considering such classification, we list all possible gradings on $UT_3^{(-)}$, up to equivalence. But beforehand we need to introduce some notation. A matrix
\[
\left(\begin{array}{cc}
g&k\\&h
\end{array}\right)
\]
will denote the elementary grading on $UT_3^{(-)}$ such that $\deg e_{12}=g$, $\deg e_{23}=h$, and $\deg e_{13}=k$. Of course $\deg e_{ii}=1$ for $i=1$, 2, 3, and necessarily $k=gh$.

Let $x$ and $y$ be generators of a free group, and consider the elementary grading $\eta_U=(x,y)$ on $\mathrm{UT}_3^{(-)}$. If $\eta=(g,h)$ is another elementary grading, then $\eta$ is a coarsening of $\eta_U$ by a group homomorphism. This follows from the map $x\mapsto g$ and $y\mapsto h$. Hence, to find all elementary gradings on $\mathrm{UT}_3^{(-)}$, up to an isomorphism, we need to find quotients of the free group.

Let $g$, $h$, $k\in G$ be three different elements, none of them trivial. The gradings on $UT_3^{(-)}$ up to equivalence are the following:\\[0cm]

\begin{tabular}{|p{3cm}|p{3cm}|p{3cm}|}
\hline
\multicolumn{3}{|p{9cm}|}{\centering Universal:\\$\displaystyle\left(\begin{array}{cc}g&k\\&h\end{array}\right)$}\\\hline
\centering Canonical:\\$\displaystyle\left(\begin{array}{cc}g&k\\&g\end{array}\right)$,\\ that is $g=h$ & {\centering Almost Universal:\\ $\displaystyle\left(\begin{array}{cc}g&1\\&h\end{array}\right)$,\\ that is $gh=1$} & {\centering Remaining:\\ $\displaystyle\left(\begin{array}{cc}g&g\\&1\end{array}\right)$,\\ that is $h=1$}\\\cline{1-2}
\multicolumn{2}{|p{6cm}|}{\centering Almost Canonical: $\displaystyle\left(\begin{array}{cc}g&1\\&g\end{array}\right)$, that is $g=h$ and $gh=1$}&\\\hline
\multicolumn{3}{|p{9cm}|}{\centering Trivial: $\displaystyle\left(\begin{array}{cc}1&1\\&1\end{array}\right)$, that is $g=1$ and $h=1$}\\\hline
\end{tabular}

\smallskip

For each one of the above elementary gradings, we compute a basis for its graded identities and its graded codimension sequence. These computations have already been known for some cases. For example the numerical invariants have been known for the trivial grading (see for instance \cite{GZ}).

\subsection{Universal grading} For the Universal grading, the graded polynomial identities and the corresponding graded codimension sequence are easy to compute, and they were obtained in \cite{pkfy2}. Denoting the Universal grading by $\eta_\text{u}$, for the particular case $n=3$, one has
\[
c_m^G(UT_3^{(-)},\eta_\text{u})=(2m^2-2m)\cdot2^{m-3}+3m,\text{ for $m\ge2$},
\]
and $c_1^G(UT_3^{(-)},\eta_\text{u})=4$.

\subsection{Almost Universal grading} Let $g\in G$ be an element of order $\ge 3$. We study the almost Universal grading, that is the elementary grading given by $\eta_\text{au}=(g,g^{-1})$.

By means of a direct verification we can see that the following polynomials are graded identities for $(UT_3^{(-)},\eta_\text{au})$:
\begin{align*}
&x_1^{(l)}=0,\quad l\not\in\{1,g,g^{-1}\},\\{}
&[x_1^{(l)},  x_2^{(l)}]=0, \quad l\in\{g,g^{-1}\},\\{}
&[y_1,y_2,x_3^{(l)}]=0, \quad l\in\{g,g^{-1}\},\\{}
&[[y_1,y_2],[y_3,y_4]]=0.
\end{align*}
Let $T_\text{au}$ be the $T_G$-ideal generated by the above polynomials. We consider monomials of the type
\[
	m_1=[x_i^{(g)},y_{i_1},y_{i_2},\ldots,y_{i_r}], m_2=[x_j^{(g^{-1})},y_{j_1},y_{j_2},\ldots,y_{j_s}]
\]
where $i_1<i_2<\ldots<i_r$ and $j_1<j_2<\ldots<j_s$. Then, using the Jacobi identity, it is easy to verify that the monomials $m_1$, $m_2$, $[m_1,m_2]$, together with the 1-convenient monomials, span the vector space of all multilinear graded polynomials modulo $T_\text{au}$. We have the following lemma. 
\begin{Lemma}\label{evaluation_au}
The above monomials are linearly independent modulo $T_G(UT_3^{(-)},\eta_\text{au})$.
\end{Lemma}
\begin{proof}
We work modulo the graded identities of $(UT_3^{(-)},\eta_\text{au})$. It is sufficient to prove the claim for homogeneous (in the grading) and multihomogeneous (in the variables) subsets of the set of all graded polynomials. Thus it is sufficient to consider only the following three cases.

\noindent\textbf{Case 1.} Linear combination of polynomials $[m_1,m_2]$.
	Let $\sum\lambda[m_1,m_2]=0$ be a multihomogeneous linear combination of  polynomials of the above type. If we consider an evaluation of the kind $x_i^{(g)}=e_{12}$, $x_j^{(g^{-1})}=e_{23}$, $y_{i_1}=\cdots=y_{i_r}=-e_{11}$, and $y_{j_1}=\cdots=y_{j_s}=e_{33}$ then there is at most one polynomial yielding non-zero evaluation, namely
\[
		[[x_i^{(g)},y_{i_1},y_{i_2},\ldots,y_{i_r}],[x_j^{(g^{-1})},y_{j_1},y_{j_2},\ldots,y_{j_s}]]=e_{13}.
\]
Hence the corresponding $\lambda=0$. Repeating the argument for the remaining summands we get that all $\lambda=0$.

\noindent\textbf{Case 2.} Linear combination of polynomials $m_1$ or $m_2$.
	Fixing a set of variables, there is at most one $m_1$ or at most one $m_2$ we can construct with the given set of variables. This case is done.

\noindent\textbf{Case 3.} 1-convenient monomials in $y_1$, $y_2$, \dots, $y_m$.
	There are $m-1$ different monomials in this case, namely $p_i=[y_i,y_1,y_2,\ldots,\hat y_i,\ldots,y_m]$ where $\hat y_i$ means that $y_i$ does not appear in the long commutator. Consider a linear combination $p=\sum_{i=1}^{m-1}\lambda_ip_i=0$. Do the following evaluation:
	\begin{align*}
		&y_i=e_{13},\\
		&y_1=\cdots=\hat y_i=\cdots =y_m=e_{33}.
	\end{align*}
	Then we obtain $p=\lambda_i e_{13}=0$ hence $\lambda_i=0$.
\end{proof}
It is easy to compute its graded codimension sequence. So, we proved the following.
\begin{Thm}\label{au_gr}
	Let $g\in G$, $o(g)\ge 3$. The graded identities of the almost Universal elementary grading $(UT_3^{(-)},(g,g^{-1}))$ follow from
	\begin{align*}
		&x_1^{(l)}=0, \quad l\not\in\{1,g,g^{-1}\},\\{}
		&[x_1^{(l)},x_2^{(l)}]=0, \quad l\in\{g,g^{-1}\},\\{}
		&[x_1^{(1)},x_2^{(1)},x_3^{(l)}]=0, \quad l\in\{g,g^{-1}\},\\{}
		&[[x_1^{(1)},x_2^{(1)}],[x_3^{(1)},x_4^{(1)}]]=0.
	\end{align*}
	In particular $c_m^G(UT_3^{(-)},(g,g^{-1}))=2m^2\cdot2^{m-3}-2m\cdot2^{m-3}+3m-1$, for $m\ge2$, and $c_1^G(UT_3^{(-)},(g,g^{-1}))=3$.\qed
\end{Thm}

\subsection{Almost Canonical grading} Here we denote $\mathbb{Z}_2=\langle 1\mid 1+1=0\rangle=\{0,1\}$. As usual, we rename also the variables $z_i=x_i^{(1)}$ for each $i\in\mathbb{N}$. Let $\eta_\text{ac}=(1,1)$ and consider the almost Canonical elementary grading $(UT_3^{(-)},\eta_\text{ac})$. It is easy to verify that the following polynomials are $\mathbb{Z}_2$-graded identities for $(UT_3^{(-)},\eta_\text{ac})$:
\begin{align*}
[y_1,y_2,z_3]=0,\\{}
[[y_1,y_2],[y_3,y_4]]=0,\\{}
[z_1,z_2,z_3]=0.
\end{align*}
Denote by $T_\text{ac}$ the $T_{\mathbb{Z}_2}$-ideal generated by the above polynomials. We call a monomial \textit{convenient} if it has the form $[z_i,y_{i_1},\ldots,y_{i_r}]$ and $i_1<\cdots<i_r$.

\begin{Lemma}
The monomials $u$, $[u_1,u_2]$ where $u$, $u_1$, $u_2$ are convenient monomials, together with the 1-convenient monomials span the vector space of all multilinear graded polynomials, modulo $T_\text{ac}$.
\end{Lemma}
\begin{proof}
	It is sufficient to write a monomial $m=[x_{i_1}^{(g_1)},\ldots,x_{i_m}^{(g_m)}]$ as a combination of polynomials in the statement of the lemma. Let $t=|\{j\mid g_j=1\}|$.

	If $t=0$ then we can use the identity $[[y_1,y_2],[y_3,y_4]]=0$, Jacobi identity and antisymmetry to obtain the claim. The argument is standard and the verification is quite easy so they are omitted.

	If $t>0$ and $g_1=g_2=0$ then, since $[y_1,y_2,z]=0$ we obtain $m=0$. Thus using antisymmetry if necessary, we can assume $g_1=1$.

	If $t\in\{1,2\}$, we apply the Jacobi identity and $[y_1,y_2,z]=0$ several times in order to obtain the claim. If $t>2$ then, as $[z_1,z_2,z_3]=0$, we obtain $m=0$. This concludes the proof.
\end{proof}

We can make evaluations similar to those in the Lemma \ref{evaluation_au} in order to obtain that the above polynomials are linearly independent modulo $T_{\mathbb{Z}_2}(UT_3^{(-)},\eta_\text{ac})$. As a consequence, we obtain
\begin{Thm}\label{ac_gr}
The $\mathbb{Z}_2$-graded identities of the almost Canonical elementary grading $(UT_3^{(-)},(1,1))$ follow from
\begin{align*}
&[x_1^{(0)},x_2^{(0)},x_3^{(1)}]=0,\\{}
&[x_1^{(1)},x_2^{(1)},x_3^{(1)}]=0,\\{}
&[[x_1^{(0)},x_2^{(0)}],[x_3^{(0)},x_4^{(0)}]]=0.
\end{align*}
In particular, $c_m^{\mathbb{Z}_2}(UT_3^{(-)},(1,1))=m^2\cdot2^{m-3}-m\cdot2^{m-3}+2m-1$, for $m\ge2$, and $c_1^{\mathbb{Z}_2}(UT_3^{(-)},(1,1))=2$.\qed
\end{Thm}

\subsection{Remaining grading} We consider the remaining elementary grading on $UT_3^{(-)}$. Let $g\in G$ be any non-trivial element and consider the elementary grading given by $\eta_r=(g,1)$. We denote $z_i=x_i^{(g)}$, for every $i\in\mathbb{N}$, and we use the notation $y_i=x_i^{(1)}$.

It is straightforward to check that the following elements are graded identities for $(UT_3^{(-)},\eta_r)$:
\begin{align*}
&[z_1,z_2]=0,\\{}
&[[y_1,y_2],[y_3,y_4]]=0,\\{}
&[z_1,[y_2,y_3],[y_4,y_5]]=0,\\{}
&x_1^{(l)}=0,\,l\not\in\{1,g\}.
\end{align*}
Denote by $T_\text{r}$ the $T_G$-ideal generated by the above polynomials.

We define two family of polynomials as follows.
\begin{Def}
Let $f$ be multilinear in the variables $z_1$, $y_1$, $y_2$, \dots, $y_h$ in the following form. Let $m\ge0$, $t\ge0$, $l\ge0$, and let
\[
f=[z_1,y_1,y_2,\ldots,y_{m+1},y_{i_1},\ldots,y_{i_t},[y_{m+2},y_j],y_{j_1},\ldots,y_{j_l}]
\]
where $i_1<\cdots<i_t$, $j_1<\cdots<j_l$. We call $f$ a \textit{kind 1 monomial}.
\end{Def}
\begin{Def}
We define \textit{kind 2 monomial} as follows. Let $f$ be a multilinear polynomial in the variables $z_1$, $y_1$, $y_2$, \dots, $y_m$ such that $f=[u,z_1,y_{j_1},\ldots,y_{j_t}]$ where $u$ is 1-convenient, $y_1$ appears in $u$, $l(u)\ge2$, $t\ge0$, and $j_1<\cdots<j_t$. We call $f$ a kind 2 monomial.
\end{Def}
Next we make a construction to formalize an extension of the definitions of kind 1 and kind 2 monomials. Let $\pi\colon \mathbb{F}\langle X_G\rangle\to\mathbb{F}\langle X_G\rangle$ be a graded algebra homomorphism such that
\begin{enumerate}
\renewcommand{\labelenumi}{(\roman{enumi})}
\item $\pi(z_1)=z_i$, for some $i\in\mathbb{N}$,
\item $\pi(y_i)=y_j$, for some $j\in\mathbb{N}$, for all $i\in\mathbb{N}$, such that $i_1<i_2$ if and only if $j_1<j_2$, where $\pi(y_{i_l})=y_{j_l},l=1,2$.
\end{enumerate}
We call also a kind 1 and 2 monomial, respectively, the monomials $\pi(f)$ where $f$ is a kind 1 or kind 2 monomial in the sense of definitions above.

\begin{Lemma}\label{r_lun}
	The 1-convenient, kind 1 and kind 2 monomials, together with the monomial $[z_i,y_{i_1},\ldots,y_{i_m}]$ with $i_1<\cdots<i_m$, span the vector space of all multilinear graded polynomials, modulo $T_\text{r}$.
\end{Lemma}
\begin{proof}
Every monomial containing only variables of trivial $G$-degree can be written as a combination of 1-convenient monomials. If a monomial contains 2 or more variables of non-trivial $G$-degree then it vanishes modulo $T_\text{r}$, since $[z_1,z_2]=0$. Thus we have to prove the lemma for a multilinear polynomial $f$ in the variables $z$, $y_1$, \dots, $y_h$. It is well known that we can assume $f$ is a linear combination of monomials of the form $[y_1,\ldots]$. In particular, we can assume that $f=[y_1,\ldots]$.

Write $f=[y_1,y_{i_1},\ldots,y_{i_r},z,y_{j_1},\ldots,y_{j_s}]$. If $r\ge 1$, we can order the last $s$ variables using the identity $[[y_1,y_2],z,[y_3,y_4]]=0$. We can focus on the first $r+1$ variables $[y_1,y_{i_1},\ldots,y_{i_r}]$ and write this commutator as a sum of 1-convenient monomials, hence we obtain that $f$ is a sum of kind 2 monomials.

If $r=0$, then, using antisymmetry, we can write $f=[z,y_1,y_{i_1},\ldots,y_{i_t}]$. We can assume that, for some $m\ge0$,
\[
f=[z,y_1,\ldots,y_{m+1},y_{j_1},\ldots,y_{j_l}]
\]
where $j_1\ne m+2$. Assume that $j_a=m+2$. Then, using the Jacobi identity, we obtain
\begin{align*}
&f=[z,y_1,\ldots,y_{m+1},y_{j_1},\ldots,y_{j_a},y_{j_{a-1}},\ldots,y_l]\\{}
&\phantom{f=}+[z,y_1,\ldots,y_{m+1},y_{j_1},\ldots,[y_{j_{a-1}},y_{m+2}],\ldots,y_{j_l}].
\end{align*}
The second is a kind 1 monomial. We apply the Jacobi identity $a-1$ times and we obtain that $f$ is the sum of $[z,y_1,\ldots,y_{m+2},y_{j_1},\ldots]$ and kind 1 monomials. Continuing the process several times, we obtain that $f$ is a linear combination of kind 1 monomials and $[z,y_1,\ldots,y_h]$, proving the lemma.
\end{proof}
\begin{Lemma}\label{r_ldeux}
	The 1-convenient, kind 1 and kind 2 monomials, and $[z_i,y_{i_1},\ldots,y_{i_m}]$, with $i_1<\cdots<i_m$ are linearly independent modulo $T_G(UT_3^{(-)},\eta_\text{r})$.
\end{Lemma}
\begin{proof}
	It is sufficient to prove the lemma for a linear combination of multihomogeneous polynomials. We can do similar considerations to prove that the 1-convenient monomials are linearly independent. Hence assume a linear combination of $[z,y_1,\ldots,y_m]$, kind 1, and kind 2 polynomials in the variables $z,y_1,y_2,\ldots,y_m$.

	Consider the evaluation $z=e_{13}$. All kind 1 and kind 2 monomials annihilate under such substitution, hence $[z,y_1,\ldots,y_m]$ does not effectively participate in such a linear combination.

	Consider evaluations such that $y_1=-e_{11}$ and $z=e_{12}$. All kind 1 monomials annihilate, regardless of the remaining variables. Now consider the evaluation $y_m=e_{23}$, $y_{m-1}=e_{22}$, and $y_2=\cdots=y_{m-2}=e_{11}$. There is unique non-vanishing kind 2 monomial, namely $[z,y_1,\ldots,y_{m-2},[y_{m-1},y_m]]$. In the remaining set of kind 2 monomials, there are not any more polynomials containing the commutator $[y_{m-1},y_m]$. Do the following two evaluations in sequence: $y_m=e_{23}$, $y_{m-2}=e_{22}$, and $y_2=\cdots=y_{m-3}=e_{11}$; then $y_{m-1}=e_{11}$, and $y_{m-1}=e_{22}$. We obtain that polynomials containing the commutator $[y_{m-2},y_m]$ do not participate in the linear combination. We can continue the process by induction, and this proves that all kind 2 monomials have zero coefficients.

	It remains to prove that the kind 1 monomials are linearly independent. Consider the evaluation $y_1=e_{22}$, $y_2=e_{23}$, $z=e_{12}$, the remaining variables assuming either $y_{22}$ or $e_{11}$. All kind 1 monomials not starting with $y_2$, must annihilate. For every choice of evaluations $(y_3,\ldots,y_m)\in\{e_{11},e_{22}\}^{m-2}$ there is unique non-vanishing kind 1 monomial. Using similar considerations for the remaining variables, we prove the lemma.
\end{proof}
The previous two lemmas compute the graded identities for the Remaining elementary grading and produce a basis for the multilinear graded polynomials, modulo the identities of $(UT_3^{(-)},\eta_\text{r})$. Therefore we can compute the graded codimension sequence.
\begin{Lemma}\label{r_ltrois}
$c_m^G(UT_3^{(-)},\eta_\text{r})=2m^2\cdot2^{m-3}-6m\cdot2^{m-3}+3m-1$, for $m\ge2$, and $c_1^G(UT_3^{(-)},\eta_\text{r})=2$.
\end{Lemma}
\begin{proof}
	Let $a\in G^m$. Then $P_m^a(UT_3^{(-)},\eta_\text{r})=0$, unless $a=(1,1,\ldots,1)$ or $a=(1,\ldots,1,a,1,\ldots,1)$. Also $\dim P_m^{(1,\ldots,1)}(UT_3^{(-)},\eta_\text{r})=m-1$, since there are $m-1$ 1-convenient monomials in the variables $y_1$, \dots, $y_m$. Moreover $\dim P_m^a(UT_3^{(-)},\eta_\text{r})=\dim P_m^{(1,\ldots,1,g)}(UT_3^{(-)},\eta_\text{r})$ for every $a=(1,\ldots,1,g,1,\ldots,1)$. Therefore
\[
		c_m^G(UT_3^{(-)},\eta_\text{r})=(m-1)+m\cdot\dim P_m^{(1,\ldots,1,g)}(UT_3^{(-)},\eta_\text{r}).
\]
	Let $z=x_m^{(g)}$. We count the number of polynomials in $P_m^{(1,\ldots,1,g)}(UT_3^{(-)},\eta_\text{r})$.
	\begin{enumerate}
		\item There is the polynomial $[z,y_1,\ldots,y_{m-1}]$.
		\item Every kind 1 monomial is of type
\[
				[z,y_1,\ldots,y_{i-1},y_{j_1},\ldots,y_{j_t},[y_i,y_l],y_{r_1},\ldots,y_{r_s}]
\]
			where $i\ge2$, $t\ge0$, $j_1<\cdots<j_t$, $r_1<\cdots<r_s$. The index $i$ runs from $2$ to $m-2$; for each $i$ we choose the index $l$; then we choose $t$ variables to be on the left of $[y_i,y_l]$ where $t$ runs from $0$ to $m-i-2$. Therefore the quantity of kind 1 monomials is
\[
				\sum_{i=2}^{m-2}(m-i-1)\sum_{t=0}^{m-i-2}{m-i-2\choose t}=m\cdot2^{m-3}-4\cdot2^{m-3}+1.
\]
		\item We count the quantity of kind 2 monomials. We just have to choose $i$ variables to be on the left of $z$ where $i$ runs from $1$ to $m-2$; and we choose one among the $i$ variables to be the first. The remaining variables are necessarily ordered. Hence the quantity of kind 2 monomials is
			\[
				\sum_{i=1}^{m-2}i{m-2\choose i}=m\cdot2^{m-3}-2\cdot2^{m-3}.
			\]
	\end{enumerate}
	All these computations prove the lemma.
\end{proof}

Lemmas \ref{r_lun} and \ref{r_ldeux} give a generating set of $G$-graded polynomial identities for the Remaining elementary grading $\eta_\text{r}$. Lemma \ref{r_ltrois} computes its graded codimension sequence. We summarize all this. 
\begin{Thm}\label{r_gr}
Let $g\in G$ be a non-trivial element. The $G$-graded identities of the Remaining elementary grading $(UT_3^{(-)},(g,1))$ follow from
\begin{align*}
&[x_1^{(g)},x_2^{(g)}]=0,\\{}
&[[x_1^{(1)},x_2^{(1)}],[x_3^{(1)},x_4^{(1)}]]=0,\\{}
&[x_1^{(g)},[x_2^{(1)},x_3^{(1)}],[x_4^{(1)},x_5^{(1)}]]=0,\\{}
&x_1^{(l)}=0, \quad l\not\in\{1,g\}.
\end{align*}
Also $c_m^{G}(UT_3^{(-)},(g,1))=2m^2\cdot2^{m-3}-6m\cdot2^{m-3}+3m-1$ for $m\ge2$, and $c_1^G(UT_3^{(-)},(g,1))=2$.\qed
\end{Thm}

\subsection{Canonical grading} Now we proceed with computing the graded codimension sequence of the canonical $\mathbb{Z}_3$-grading on $UT_3^{(-)}$. Denote $\mathbb{Z}_3=\{0,1,2\}$ with additive notation. It is well known that the $\mathbb{Z}_3$-graded identities of $(UT_3^{(-)},(1,1))$ follow from $[x_1^{(0)},x_2^{(0)}]=0$ and $[x_1^{(i)},x_2^{(j)}]=0$ whenever $i+j\ge 3$ (see \cite{pky}). Hence in more concrete terms, these graded identities are:
\begin{align*}
&[x_1^{(0)},x_2^{(0)}]=0,\\{}
&[x_1^{(1)},x_2^{(2)}]=0,\\{}
&[x_1^{(2)},x_2^{(2)}]=0.
\end{align*}
By means of standard computations and direct evaluations, we can prove that a vector space basis of the multilinear graded polynomials, modulo the $\mathbb{Z}_3$-graded identities of the canonical grading, consists of:
\begin{align*}
	&[x_i^{(l)},x_{i_1}^{(0)},\ldots,x_{i_m}^{(0)}], \quad l\in\{1,2\}, \quad i_1<\cdots<i_m,\\{}
	&[x_i^{(1)},x_{i_1}^{(0)},\ldots,x_{i_m}^{(0)},x_j^{(1)},x_{j_1}^{(0)},\ldots,x_{j_t}^{(0)}], \quad i<j, \quad i_1<\cdots<i_m, \quad j_1<\cdots<j_t.
\end{align*}
Hence we can compute explicitly the graded codimension sequence.
\begin{Prop}\label{c_codim}
Consider the canonical $\mathbb{Z}_3$-grading on $UT_3^{(-)}$. Then its graded codimension sequence is $c_1^{\mathbb{Z}_3}(UT_3^{(-)},(1,1))=3$, and
\[
	c_m^{\mathbb{Z}_3}(UT_3^{(-)},(1,1))=m^2\cdot2^{m-3}-m\cdot2^{m-3}+2m,\text{ for $m\ge2$}.
\]\qed
\end{Prop}

\subsection{Conclusions} Given $\eta\in G^m$, denote by $|\eta|$ the quantity of different elements appearing in $\eta$. We can now use the results of Theorems \ref{au_gr}, \ref{ac_gr}, \ref{r_gr}, and Proposition \ref{c_codim}, together with the already known graded identities and graded codimension sequence for the remaining elementary gradings, in order to prove that Conjecture \ref{firstconj} is true for the particular case $UT_3^{(-)}$. Recall that, given maps $f$, $g:\mathbb{N}\to\mathbb{N}$, we denote $f\sim g$ if $\lim_{n\to\infty}f(n)/g(n)=1$.
\begin{Thm}
	Let $G$ be an abelian group and let $\eta\in G^2$ be any sequence. Then the $G$-graded identities of the elementary grading $(UT_3^{(-)},\eta)$ follow from $f_\tau$ where $\tau$ is an $\eta$-bad tree with $l(\tau)\le3$. Also the graded codimension sequence satisfies
\[
		c_m^G(UT_3^{(-)},\eta)\sim|\eta|\cdot c_m(UT_3^{(-)})\sim|\eta|\cdot m^2\cdot2^{m-3}.
\]\qed
\end{Thm}

Now we observe another interesting relation among the graded codimension sequences. Let $A$ be a graded algebra, and let $\Gamma_1$ and $\Gamma_2$ be two finite $G$-gradings on $A$. Assume $\Gamma_1$ a coarsening of $\Gamma_2$. Using the same argument as in \cite{bagiri}, we can conclude that $c_m(\Gamma_1)\le c_m(\Gamma_2)$ (see also \cite{pkfy2}). The graded codimension sequences of the elementary gradings on $UT_3^{(-)}$ indeed satisfy such relations. We call 1-kill-coarsening a coarsening of gradings which is as follows. 
\[
	\left(\begin{array}{cc}g&k\\&h\end{array}\right)\longrightarrow
	\left(\begin{array}{cc}g&1\\&h\end{array}\right)
\]
where we lose only one homogeneous element of a basis to the trivial component of the grading. For elementary gradings on $UT_3^{(-)}$ we have exactly two 1-kill-coarsenings, namely: the almost Universal and the Universal gradings, and the almost Canonical and the Canonical gradings. Denote by $\eta_\text{c}$ the Canonical grading on $UT_3^{(-)}$. According to our computations, the graded codimension sequences satisfy, for all $m\ge1$
\begin{align*}
&c_m(UT_3^{(-)},\eta_\text{u})=c_m(UT_3^{(-)},\eta_{\text{au}})+1,\\
&c_m(UT_3^{(-)},\eta_\text{c})=c_m(UT_3^{(-)},\eta_{\text{ac}})+1.
\end{align*}

Let us consider $UT_2^{(-)}$. It is well known that its ordinary polynomial identities follow from $[[x_1,x_2],[x_3,x_4]]$, and its ordinary codimension sequence satisfies $c_m(UT_2^{(-)})=m-1$, for every $m\ge2$, $c_1(UT_2^{(-)})=1$. Let $\eta_2=(g)$ be the unique non-trivial elementary grading on $UT_2^{(-)}$. Then its graded polynomial identities follow from $[x_1^{(g)},x_2^{(g)}]=0$ and $[x_1,x_2]=0$. It is easy to see that its graded codimension sequence satisfies $c_m^G(UT_2^{(-)},\eta_2)=m$ for every $m\ge2$, and $c_1^G(UT_2^{(-)},\eta_2)=2$.

Note that the trivial grading on $UT_2^{(-)}$ and the $\eta_2$ grading correspond to an 1-kill-coarsening. Also $c_m^G(UT_2^{(-)},\eta_2)=c_m(UT_2^{(-)})+1$, for every $m\ge1$. Hence we have good reasons to raise the following conjecture. 
\begin{Conj}\label{thirdconj}
	Let $\eta_1$ and $\eta_2$ be two gradings on $UT_n^{(-)}$ such that $\eta_1$ and $\eta_2$ corresponds to an 1-kill-coarsening. Then $c_m(UT_n^{(-)},\eta_2)=c_m(UT_n^{(-)},\eta_1)+1$.
\end{Conj}

\section{Type 2 Gradings on $UT_3^{(-)}$}
It is known that there exist non-elementary gradings on $UT_n^{(-)}$ \cite{KY}. Let us recall that this phenomenon does not appear in the associative case of $UT_n$. But such gradings do appear in the Lie and in the Jordan case. The description of these gradings was given in \cite{KY, pkfy_jordan}. We also recall that these gradings are not only non-elementary but they are not equivalent to any elementary grading. 

Now we will study the gradings that arise from the natural involution of $UT_n$ which is represented as the flip along the second diagonal, denoted by $T$. To this end, let $G$ be an abelian group and consider an elementary $G$-grading $(UT_n,\eta)$ such that $\eta=\text{rev}\,\eta$. In this case it is easy to see that $T\colon UT_n\to UT_n$ is a graded map. This gives rise to a decomposition of $UT_n$ into the vector spaces of homogeneous symmetric elements and homogeneous skew-symmetric elements. Such a decomposition induces a structure of $G\times\mathbb{Z}_2$-graded algebra on $UT_n^{(-)}$. We call these gradings \textsl{type 2 gradings}.

Considering the previous table of elementary gradings on $UT_3^{(-)}$, we see that there are, up to equivalence, exactly three type 2 gradings, namely: the Canonical T2, the Almost Canonical T2 and the Trivial T2 gradings.

For the purposes of facilitating computations, we will introduce some notations. Denote $e_{i:m}=e_{i,i+m}$, $e_{-i:m}=e_{n-i+1-m,n-i+1}$, $X_{i:m}=e_{i:m}-e_{i:m}$, and $X_{i:m}'=e_{i:m}+e_{i:m}$. A type 2 grading is characterized by a grading such that all $X_{i:m}$ and $X_{i:m}'$ are homogeneous and $(\deg X_{i:m})(\deg X_{i:m}')^{-1}=t$. Here $t$ a fixed element in $G$ of order 2. 

\subsection{Canonical T2 grading} Denote 
\[
\mathbb{Z}_3\times\mathbb{Z}_2=\langle1,t\mid 1+t=t+1,1+1+1=0,t+t=0\rangle.
\] 
The Canonical T2 grading is given by $UT_3^{(-)}=\sum A_g$ where
\begin{align*}
&A_0=\text{Span}\{X_{1:0}\},&A_t=\text{Span}\{X_{1:0}',X_{2:0}'\},\\{}
&A_1=\text{Span}\{X_{1:1}\},&A_{1+t}=\text{Span}\{X_{1:1}'\},\\{}
&A_2=0,&A_{2+t}=\text{Span}\{X_{1:2}'\}.
\end{align*}
By direct computation, we can prove that the following are graded identities for the Canonical T2 grading:
\begin{align*}
&[x_1^{(l)},x_2^{(l)}]=0,\quad l\in\mathbb{Z}_3\times\mathbb{Z}_2,\\{}
&[x_1^{(0)},x_2^{(t)}]=0,\\{}
&[x_1^{(l)},x_2^{(2+t)}]=0,\quad l\in\{t,1,1+t\},\\{}
&2[x_1^{(1)},x_2^{(0)},x_3^{(1+t)}]=[x_1^{(1)},x_3^{(1+t)},x_2^{(0)}],\\{}
&x_1^{(2)}=0.
\end{align*}
We draw the readers' attention to the graded identity \[
2[x_1^{(1)},x_2^{(0)},x_3^{(1+t)}]=[x_1^{(1)},x_3^{(1+t)},x_2^{(0)}]
\] 
which is non-monomial. The existence of this graded identity shows that the graded identities of the Canonical T2 grading on $UT_3^{(-)}$ cannot follow from the special-monomial identities. 

Therefore the analogue of Conjecture \ref{secondConj} is false for type 2 gradings.

\begin{Lemma}
The following polynomials constitute a vector space basis of the multilinear graded polynomials, modulo the identities of the Canonical T2 grading:
\begin{align*}
&[x_i^{(l)},x_{i_1}^{(0)},\ldots,x_{i_m}^{(0)},x_{j_1}^{(t)},\ldots,x_{j_r}^{(t)}], \quad m\ge0, r\ge 0,\\{}
&[x_i^{(2+t)},x_{i_1}^{(0)},\ldots,x_{i_m}^{(0)}],\quad m\ge0,\\{}
&[x_i^{(1)},x_{i_1}^{(0)},\ldots,x_{i_m}^{(0)},x_{j_1}^{(t)},\ldots,x_{j_r}^{(0)},x_j^{(1+t)}],\quad m\ge0,\text{$t\in\mathbb{N}\cup\{0\}$ even},\\{}
&[x_i^{(l)},x_{i_1}^{(0)},\ldots,x_{i_m}^{(0)},x_{j_1}^{(t)},\ldots,x_{j_r}^{(0)},x_j^{(l)}],\quad i<j, m\ge0, \text{$t\in\mathbb{N}$ odd}.
\end{align*}
In all these polynomials, we always assume $l\in\{1,1+t\}$, $i_1<\cdots<i_m$,  and $j_1<\cdots<j_r$.
\end{Lemma}
\begin{proof}
	Let $m$ be a monomial of the type $[x_1^{(g_1)},\ldots,x_a^{(g_a)}]$. Let $c_g$ be the number of variables of degree $g$, for $g\in\{1,1+t,2+t\}$, and let $c=c_1+c_{1+t}+c_{2+t}$.

	If $c=0$, then $m=0$ unless $m$ is a variable. If $c>0$ then necessarily the first variable (or the second) have degree $1$, $1+t$ or $2+t$. Also one must have $c\le2$.

	If $c=1$ then we can easily order the remaining variables.

	Now assume $c=2$. In this case necessarily $c_{2+t}=0$. Write
\[
m=[x_i^{(l_1)},x_{i_1}^{(g_1)},\ldots,x_{i_r}^{(g_r)},x_j^{(l_2)},x_{j_1}^{(h_1)},\ldots,x_{j_s}^{(h_s)}],
\]
with $l_1$, $l_2\in\{1,1+t\}$. Note that $m\ne0$ implies $\{l_1+g_1+\cdots+g_r,\l_2\}=\{1,1+t\}$. This implies also $h_1=\cdots=h_s=0$ since $[x^{(2+t)},x^{(t)}]=0$. Using the identity $2[x_1^{(1)},x_2^{(0)}, x_3^{(1+t)}] =[x_1^{(1)},x_3^{(1+t)},x_2^{(0)}]$, we can write
\[m=[x_i^{(l_1)},x_{i_1}^{(g_1)},\ldots,x_{i_r}^{(g_r)},x_{j_1}^{(h_1)},\ldots,x_{j_s}^{(h_s)},x_j^{(l_2)}],
\]
and the middle variables can be ordered. The conclusion that the above polynomials span the vector space of all multilnear graded polynomials, modulo the $T_{\mathbb{Z}_3\times\mathbb{Z}_2}$-ideal generated by the above graded polynomial identities, is now immediate.

	Since in the above set of polynomials there is at most one polynomial corresponding to each choice of set of variables, then clearly these are linearly independent modulo the graded polynomial identities of the Canonical T2 grading.
\end{proof}

As a consequence, we obtain the following.
\begin{Thm}\label{c_mt}
Consider the Canonical T2 $\mathbb{Z}_3\times\mathbb{Z}_2$-grading on $UT_3^{(-)}$. Then its graded polynomial identities follow from
\begin{align*}
&[x_1^{(l)},x_2^{(l)}]=0,\quad l\in\mathbb{Z}_3\times\mathbb{Z}_2,\\{}
&[x_1^{(0)},x_2^{(t)}]=0,\\{}
&[x_1^{(l)},x_2^{(2+t)}]=0,\quad l\in\{t,1,1+t\},\\{}
&2[x_1^{(1)},x_2^{(0)},x_3^{(1+t)}]=[x_1^{(1)},x_3^{(1+t)},x_2^{(0)}],\\{}
&x_1^{(2)}=0.
\end{align*}
In particular, $c_m^{\mathbb{Z}_3\times\mathbb{Z}_2}(UT_3^{(-)})=4m^2\cdot2^{m-3}+4m\cdot2^{m-3}+m$, for $m\ge2$, and $c_1^{\mathbb{Z}_3\times\mathbb{Z}_2}(UT_3^{(-)})=5$.\qed
\end{Thm}

\subsection{Almost Canonical T2 grading} We investigate now the almost Canonical T2 grading. Denote $\mathbb{Z}_2\times\mathbb{Z}_2=\langle1,t\mid1+t=t+1,1+1=0,t+t=0\rangle$. The almost Canonical T2 grading is given by its homogeneous components:
\begin{align*}
&A_0=\text{Span}\{X_{1:0}\},&A_t=\text{Span}\{X_{1:0}',X_{2:0}',X_{1:2}'\},\\{}
&A_1=\text{Span}\{X_{1:1}\},&A_{1+t}=\text{Span}\{X_{1:1}'\}.
\end{align*}
One verifies directly that the following are graded polynomial identities for this grading:
\begin{align*}
&[x_1^{(l)},x_2^{(l)}]=0, \quad l\in\mathbb{Z}_2\times\mathbb{Z}_2,\\{}
&[x_1^{(1)},x_2^{(1+t)},x_3^{(0)}]=2[x_1^{(1)},x_3^{(0)},x_2^{(1+t)}],\\{}
&[x_1^{(0)},x_2^{(t)},x_3^{(l)}]=0, \quad l\in\{1,t,1+t\},\\{}
&[x_1^{(1)},x_2^{(1+t)},x_3^{(l)}]=0, \quad l\in\{1,t,1+t\}.
\end{align*}
Let $T_\text{acmt}$ be the $T_{\mathbb{Z}_2\times\mathbb{Z}_2}$-ideal generated by above polynomials.
\begin{Lemma}
The polynomials
\begin{align*}
&[x_1^{(t)},x_{i_1}^{(0)},\ldots,x_{i_m}^{(0)}],\\{}
&[x_1^{(l)},x_{i_1}^{(0)},\ldots,x_{i_m}^{(0)},x_{j_1}^{(t)},\ldots,x_{j_s}^{(t)}],\\{}
&[x_i^{(1)},x_{i_1}^{(0)},\ldots,x_{i_m}^{(0)},x_{j_1}^{(t)},\ldots,x_{j_s}^{(t)},x_j^{(1+t)}],\text{ $s$ even,}\\{}
&[x_a^{(l)},x_{i_1}^{(0)},\ldots,x_{i_m}^{(0)},x_{j_1}^{(t)},\ldots,x_{j_s}^{(t)},x_b^{(l)}],\text{ $s$ odd,}\\{}
&l\in\{1,1+t\},i_1<\cdots<i_m,j_1<\cdots<j_s,\quad a<b,
\end{align*}
span the vector space of all multilinear graded polynomials, modulo $T_\text{acmt}$.
\end{Lemma}
\begin{proof}
	Let $m$ be a multilinear polynomial in the variables $x_1^{(0)}$, \dots, $x_{m_0}^{(0)}$, $x_1^{(t)}$, \dots, $x_{m_t}^{(t)}$, $x_1^{(1)}$, \dots, $x_{m_1}^{(1)}$, and $x_1^{(1+t)}$, \dots, $x_{m_{1+t}}^{(1+t)}$. Let $c=m_1+m_{1+t}$. Note that necessarily $c<3$, otherwise $m=0$.

	If $c=0$ then $m$ is a variable; or necessarily $m_t=1$ and the variable $x_1^{(t)}$ must be at first or second position. We can order the variables of trivial degree, if any.

	If $c=1$, we can assume $m=[x_1^{(l)},x_{i_1}^{(g_1)},\ldots,x_{i_r}^{(g_r)}]$, $g_1$, \dots, $g_r\in\{0,t\}$. Using the identity $[x^{(0)},x^{(t)},x^{(l)}]=0$, we can change the position of the variables of degree $0$ and $t$ and write $m=[x_1^{(l)},x_{i_1}^{(0)},\ldots,x_{i_p}^{(0)},x_{j_1}^{(t)},\ldots,x_{j_q}^{(t)}]$. We can order the variables of degree $0$ and of degree $t$.

	If $c=2$ we can use a similar idea to that in the Canonical T2 grading case.
\end{proof}

Using the same considerations as in the case of the Canonical T2 grading, we see that the above polynomials are linearly independent, modulo the graded identities of the almost Canonical T2 grading. As a consequence we obtain a basis of the graded polynomial identities for this grading. We can also compute the corresponding graded codimension sequence.
\begin{Thm}
	Consider the almost Canonical T2 $\mathbb{Z}_2\times\mathbb{Z}_2$-grading. Then its graded polynomial identities follow from
	\begin{align*}
		&[x_1^{(l)},x_2^{(l)}]=0, \quad l\in\mathbb{Z}_2\times\mathbb{Z}_2,
		\\{}
	&[x_1^{(1)},x_2^{(1+t)},x_3^{(0)}]=2[x_1^{(1)},x_3^{(0)},x_2^{(1+t)}],\\{}
		&[x_1^{(0)},x_2^{(t)},x_3^{(l)}]=0, \quad l\in\{1,t,1+t\},\\{}
		&[x_1^{(1)},x_2^{(1+t)},x_3^{(l)}]=0, \quad l\in\{1,t,1+t\}.
	\end{align*}
	In particular, $c_m^{\mathbb{Z}_2\times\mathbb{Z}_2}(UT_3^{(-)})=4m^2\cdot2^{m-3}+4m\cdot2^{m-3}+m$, for every $m\ge2$, and $c_1^{\mathbb{Z}_2\times\mathbb{Z}_2}(UT_3^{(-)})=4$.\qed
\end{Thm}
We observe that the almost Canonical T2 grading and the Canonical T2 grading almost constitute an 1-kill-coarsening pair. The difference here is that we send an element of the basis of the first algebra to degree $t$ instead of degree $0$ in the coarsening. But the relation $c_m^{\mathbb{Z}_3\times\mathbb{Z}_2}(UT_3^{(-)})=c_m^{\mathbb{Z}_2\times\mathbb{Z}_2}(UT_3^{(-)})+1$ is not true, unless $m=1$. 

As a conclusion, Conjecture \ref{thirdconj} is not true for T2 gradings.

\subsection{Trivial T2 grading} Denote $\mathbb{Z}_2=\langle t\mid t+t=0\rangle=\{0,t\}$. The last grading to consider is the Trivial T2 grading. This grading is given by the following decomposition into homogeneous subspaces:
\begin{align*}
&A_0=\text{Span}\{X_{1:0},X_{1:1}\},\\
&A_t=\text{Span}\{X_{1:0}',X_{2:0}',X_{1:1}',X_{1:2}'\}.
\end{align*}
Denote by $y_i=x_i^{(0)}$ and $z_i=x_i^{(t)}$, for every $i\in\mathbb{N}$. Consider the action of the symmetric group $S_m$ permuting the variables $z$.
\begin{Lemma}
	The following are graded identities for the Trivial T2 grading:
	\begin{align*}
		&[u_1,u_2]=0,\\{}
		&[v_1,v_2]=0,\\{}
		&[u_1,v_2,z_5]=0,\\{}
		&2[u_1,y_5,v_2]=[u_1,v_2,y_5],\\{}
		&2\tau\cdot[y_1,z_1,z_2,y_2,z_3]=\tau\cdot[y_1,z_1,z_2,z_3,y_2],\\{}
		&\sum_{\sigma\in S_3}(-1)^\sigma\sigma\cdot[y_1,z_1,z_2,z_3]=0
	\end{align*}
	where $v_i=[y_{2i-1},z_{2i}]$ and $u_i$ is either equal to $[y_{2i-1},y_{2i}]$ or $[z_{2i-1},z_{2i}]$, and $\tau=(1-(2\quad3))\in\mathbb{F}S_3$.
\end{Lemma}
\begin{proof}
The proof is a direct verification and hence is omitted.
\end{proof}

Let $T_\text{tmt}$ be the $T_{\mathbb{Z}_2}$-ideal generated by above polynomials. We do not know if this list gives a basis for the graded polynomial identities of the Trivial T2 grading. However, we give a partial result.

Consider polynomials of type
\begin{equation}\label{monom_special}
[w,z_{i_1},\ldots,z_{i_r},y_{j_1},\ldots,y_{j_s}], \quad i_1<\cdots<i_r, \quad j_1<\cdots<j_s.
\end{equation}
Here $w=[y_j,x_i^{(g)}]$, with $j$ the lowest index of $x$ appearing in $m$, and $x_i^{(g)}$ is any variable.
\begin{Lemma}
	Consider a polynomial $f$ with $\deg_{\mathbb{Z}_2}f=0$ and $f$ containing 4 or more variables, and at least one even variable. Then $f$ can be written as a linear combination of monomials of type (\ref{monom_special}).
\end{Lemma}
\begin{proof}
	Let $m$ be a monomial in some variables $y$'s and $z$'s, with $\deg_{\mathbb{Z}_2}m=0$. We can assume that $m=[x_{a_1}^{(0)},x_{a_2}^{(g_2)},\ldots,x_{a_b}^{(g_b)}]$ where $a_1$ is the lowest variable $x$ appearing in $m$. Denote by $w=[y_{a_1},x_{a_2}^{(g_2)}]$. Now we only have to order the remaining variables and change the positions of some $y$'s and $z$'s. The following are graded polynomials in $T_\text{tmt}$:
	\begin{enumerate}
		\renewcommand{\labelenumi}{(\roman{enumi})}
		\item $[w,y,z_1,z_2]=[w,z_1,y,z_2] \text{ where $w=u$ or $w=v$}$,
		\item $[v,y,z]=[v,z,y]$.
	\end{enumerate}
	Hence we obtain $m=[w,z_{i_1},\ldots,z_{i_r},y_{j_1},\ldots,y_{j_r}]$. Since $\mathbb{Z}_2\text{-deg}\,m=0$ we can easily write $[w,z_{i_1},\ldots,z_{i_r}]$ as a commutator like $u$. Using the identity $[u_1,u_2]=0$, we can order the last $x$'s.

	Now using $[u,v,z]=0$, we can order the first $i_r-1$ variables $z$'s. Using again that $\deg_{\mathbb{Z}_2}m=0$, we can see $[w,z_{i_1},\ldots,z_{i_r-2}]$ as an element $u$, so
\[
[w,z_{i_1},\ldots,z_{i_r-2},z_{i_r-1},z_{i_r}]=[w,z_{i_1},\ldots,z_{i_r-2},z_{i_r},z_{i_r-1}].
\]
This proves that we can order the $z$'s, concluding the proof.
\end{proof}

It is easy to prove that the polynomials (\ref{monom_special}) are linearly independent modulo the $\mathbb{Z}_2$-graded identities of the Trivial T2 grading. To this end, consider a multilinear polynomial in the variables $y_1$, \dots, $y_s$, $z_1$, \dots, $z_s$, of type (\ref{monom_special}). Consider the following evaluation. Choose some index $a$ and put $y_a=X_{1:1}$. Evaluate the remaining variables as $y_i=X_{1:0}$ and $z_j=X_{1:0}'$. There is unique polynomial producing a non-zero evaluation, namely $[y_1,y_a,z,\ldots,z,y,\ldots,y]$. Now take $z_b=X_{1:1}'$, and take the remaining variables once again $y_i=X_{1:0}$ and $z_j=X_{1:0}'$. There is again unique polynomial giving non-zero evaluation, namely $[y_1,z_b,z,\ldots,z,y,\ldots,y]$. This proves the claim.


\end{document}